\documentclass[12pt]{amsart}
\usepackage{amsmath, amssymb, amsthm, amsfonts, mathrsfs}

\newtheorem{proposition}{Proposition}
\newtheorem{definition}{Definition}
\newtheorem{lemma}[proposition]{Lemma}

\newtheorem{theorem}[proposition]{Theorem}

\newtheorem{corollary}[proposition]{Corollary}
\theoremstyle{definition}

\theoremstyle{definition}

\numberwithin{proposition}{section}

\DeclareSymbolFont{AMSb}{U}{msb}{m}{n}
\DeclareMathSymbol{\N}{\mathbin}{AMSb}{"4E}
\DeclareMathSymbol{\Z}{\mathbin}{AMSb}{"5A}
\DeclareMathSymbol{\R}{\mathbin}{AMSb}{"52}
\DeclareMathSymbol{\Q}{\mathbin}{AMSb}{"51}
\DeclareMathSymbol{\I}{\mathbin}{AMSb}{"49}
\DeclareMathSymbol{\C}{\mathbin}{AMSb}{"43}

\newtheorem{fact}{Fact}
\newtheorem{claim}{Claim}

\title[Centralizers of Rank-1 Homeomorphisms]{Centralizers of Rank-1 Homeomorphisms}
\author{Aaron Hill}

\begin{document}
\maketitle

\begin{abstract}
We give a definition for a rank-1 homeomorphism of a zero-dimensional Polish space $X$.  We show that if a rank-1 homeomorphism of $X$ satisfies a certain non-degeneracy condition, then it has trivial centralizer in the group of all homeomorphisms of $X$, i.e., it commutes only with its integral powers. 
\end{abstract}

\section{Introduction}

Let Aut$(X,\mu)$ denote the group of invertible measure-preserving transformations of a standard Lebesgue space $(X,\mu)$, taken modulo null sets and equipped with the weak topology.  The subset of Aut$(X,\mu)$ consisting of rank-1 transformations has been extensively studied.  One of the important results about rank-1 transformations is King's weak closure theorem \cite{King1}:  

\begin{theorem}[King, 1986]  
If $T \in \textnormal{Aut}(X,\mu)$ is rank-1, then the centralizer of $T$ in the group Aut$(X,\mu)$ equals $\overline{\{T^i: i \in \mathbb{Z}\}}$.  
\end{theorem}


Some rank-1 measure-preserving transformations commute only with their integral powers (e.g., Chacon's transformation, see \cite{delJunco}), but this is not typical.  It is well known that a generic measure-preserving transformation $T$ is rigid, and thus  $\overline{\{T^i : i \in \Z\}}$ is uncountable.  Stepin and Ereminko showed in \cite{StepinEremenko} that every compact abelian group embeds into the centralizer of a generic measure-preserving transformation.  Glasner and Weiss showed in \cite{GlasnerWeiss} that for generic $T$, there is no spatial realization of $\overline{\{T^i : i \in \Z\}}$, which, by a theorem of Mackey, implies that $\overline{\{T^i : i \in \Z\}}$ is not locally compact.  Even more recently, Mellerey and Tsankov showed that for generic $T$, the group $\overline{\{T^i : i \in \Z\}}$ is extremely amenable.  As a generic measure-preserving transformation is rank-1, these results imply, in a very strong way, that for rank-1 transformations, $\overline{\{T^i: i \in \mathbb{Z}\}}$ typically contains much more than $\{T^i: i \in \mathbb{Z}\}$.


In this paper we introduce the notion of a rank-1 homeomorphism of a zero-dimensional Polish space $X$.  Our main result is that if such a rank-1 homeomorphism $f$ satisfies a certain non-degeneracy condition, it has trivial centralizer in the group of homeomorphisms of $X$, i.e., the centralizer of $f$ in the group Homeo$(X)$ equals $\{f^i : i \in \mathbb{Z}\}$.  

To motivate this, we will first describe how homeomorphisms can naturally be obtained from rank-1 measure-preserving transformations.  In the literature, there are several different definitions of rank-1 transformations; these are all equivalent if we restrict our attention to transformations that are totally ergodic, i.e. those $T$ for which $T^n$ is ergodic for each $n>0$.  A nice discussion of these several definitions and their interconnections can be found in the survey article \cite{Ferenczi}.  Below we mention two ways of defining rank-1 transformations and how one can realize rank-1 measure-preserving transformations as homeomorphisms of zero-dimensional Polish spaces.

One way of defining rank-1 transformations is via symbolic systems.  One first defines a collection of symbolic rank-1 systems.  Each such system is a triple $(X,\mu, \sigma)$, where $X$ is some closed, shift-invariant subset of $2^\mathbb{Z}$ with no isolated points, $\mu$ is a shift-invariant, non-atomic probability measure supported on $X$, and $\sigma$ is the shift.  A totally ergodic measure-preserving transformation is rank-1 if it is isomorphic to some symbolic rank-1 system.  For a symbolic rank-1 system, the shift $\sigma$ is not just a measure-preserving transformation, but also a homeomorphism of the Cantor space $X$.  King's theorem gives us information about the centralizer of $\sigma$ in the group Aut$(X,\mu)$, but one can also ask about the centralizer of $\sigma$ in the group Homeo$(X)$.  It is a consequence of our main theorem that $\sigma$ has trivial centralizer in the group Homeo$(X)$.

Another way of defining rank-1 transformations is via ``cutting and stacking" transformations.  In this case, one defines a collection $C$ of measure-preserving transformations, each obtained by a cutting and stacking construction with intervals.  A measure-preserving transformation is said to be rank-1 if it is isomorphic to some element of $C$.  Each $T \in C$ naturally gives rise to a homeomorphism as follows:  In the cutting and stacking construction of $T$, one removes from the interval $[0,1]$ both the point 0 and the point 1 and each cut-point of the cutting and stacking procedure.  By doing this, one removes a countable dense subset from the interval $(0,1)$, and thus what remains (call it $Y$) is homeomorphic to Baire space.  The transformation $T$ restricted to $Y$ is still a rank-1 measure-preserving transformation, but it is also a homeomorphism of $Y$.  It is a consequence of our main theorem that if $T \in C$ is totally ergodic, then the corresponding homeomorphism has trivial centralizer in the group Homeo$(Y)$.     

The natural topological analogue of ergodicity is this:  A homeomorphism $f$ of a Polish space $X$ is {\em transitive} if for non-empty open sets $U$ and $V$, there is some $i \in \Z$ so that $f^i(U)$ intersects $V$.  This clearly implies that for each non-empty open set $U$, the set $\bigcup_{i \in \mathbb{Z}} f^i(U)$ is co-meager in $X$.  A homeomorphism of $X$ is {\em totally transitive} if $f^n$ is transitive for each $n>0$.  The non-degeneracy condition of the theorem will be satisfied by any homeomorphism that is totally transitive (and some that are are not totally transitive).


It should be noted that in \cite{Bezuglyi} there is a definition given for a rank-1 homeomorphism of a Cantor space and several results are proved about homeomorphisms that satisfy that definition.  Their analysis is complementary to the analysis of this paper in that their definition for rank-1 is shown in their paper to be equivalent to being conjugate to an odometer, and such homeomorphisms do not satisfy the non-degeneracy condition that our theorem requires.

\section{Preliminaries}

While the definition given below is for any Polish space $X$, it should be noted that the existence of a rank-1 homeomorphism of $X$ implies, by condition (5), that $X$ has a basis of clopen sets, i.e., it is zero-dimensional.  

\begin{definition}
A homeomorphism $f$ of a Polish space $X$ is {\em rank-1} if there exists a sequence $\{B_n\}$ of strictly decreasing clopen sets and a sequence $\{h_n\}$ of strictly increasing positive integers so that: 

\begin{enumerate}
\item  the sets $B_n$, $f(B_n)$, \dots, $f^{h_n -1}(B_n)$ are pairwise disjoint;
\item  $\bigcup _{0 \leq i < h_n} f^i(B_n) \subseteq \bigcup _{0 \leq j < h_{n+1}} f^j(B_{n+1})$;
\item  if $f^j (B_{n+1}) \cap f^i (B_n) \neq \emptyset$ and $0 \leq j < h_{n+1}$ and $0 \leq i < h_n$, then $f^j (B_{n+1}) \subseteq f^i (B_n);$ 
\item  $f^{h_{n+1} -1} (B_{n+1}) \subseteq f^{h_n -1} (B_n)$; and
\item  the orbits of the sets $B_n$ under $f$ and orbits of the sets $L_n := X \setminus \bigcup_{0 \leq i < h_n} f^i (B_n)$ under $f$ form a basis for the topology of $X$.
\end{enumerate}
\end{definition}

We will first establish some terminology for rank-1 homeomorphisms.  If $f$ is a homeomorphism of a Polish space $X$, and the sequences $\{B_n\}$ and $\{h_n\}$ witness that $f$ is a rank-1 homeomorphism, then we call the pair $(\{B_n\}, \{h_n\})$ a {\em tower representation} of $f$.  Every rank-1 homeomorphism has multiple tower representations.  For example, if $(\{B_n\}, \{h_n\})$ is a tower representation of $f$, then by removing a single entry $B_k$ from the sequence $\{B_n\}$ and the corresponding entry $h_k$ from the sequence $\{h_n\}$, one obtains another tower representation of $f$.

Let $f$ be a rank-1 homeomorphism of a Polish space with a fixed tower representation $(\{B_n\}, \{h_n\})$.  For $n \in \N$ and $0 \leq i < h_n$, the set $f^i(B_n)$ is called the $i$-th level of the stage-$n$ tower.  We also call $B_n$ the {\em base} of the stage-$n$ tower and call $f^{h_n -1} (B_n)$ the {\em top} of the stage-$n$ tower.  We say that $h_n$ is the {\em height} of the stage-$n$ tower and call $L_n$ the {\em leftover piece} of the stage-$n$ tower. 

The definition of a rank-1 homeomorphism requires that each level of each tower be either a subset of some level of the stage-0 tower or a subset of $L_0$.  Borrowing terminology from the measure-preserving situation, we sometimes refer to levels of towers that are contained in $L_0$ as {\em spacers}.  Define $W_n\colon h_n \rightarrow \{0,1\}$ so that $W_n(i) = 0$ iff $f^i (B_n)$ is contained in some level of the stage-0 tower.  Since the sequences $\{B_n\}$ and $\{f^{h_n -1} (B_n)\}$ are each decreasing, each $W_{n+1}$ begins and ends with an occurrence of $W_n$.  In particular, each $W_{n+1}$ begins and ends with 0.  We can thus define $W_\infty \colon \N \rightarrow \{0,1\}$ so that for all $n\in \N$, $W_\infty (i) = W_n (i)$ whenever $W_n (i)$ is defined.  There are three possibilities:

\begin{enumerate}
\item  The sequence $W_\infty$ is periodic.  In this case, we say that the tower representation of $f$ is {\em repeating}.  One can show that in this case $f$ is not totally transitive.
  
Odometers form an important class of measure-preserving transformations, and each can be realized as a rank-1 homeomorphism of a Cantor space with a repeating tower representation (in fact, with $L_0 = \emptyset$).

\item  The sequence $W_\infty$ is not periodic, but there is a bound on the number of consecutive $1$s in $W_\infty$.  In this case, we say that the tower representation of $f$ is {\em non-repeating} and furthermore that it has bounded sequences of consecutive spacers.  

The symbolic version of Chacon's (measure-preserving) transformation is a homeomorphism of a Cantor space.  The natural choice for a tower representation of this homeomorphism witnesses that it is rank-1 and has bounded sequences of consecutive spacers.

\item  There are arbitrarily long sequences of consecutive $1$s in $W_\infty$.  In this case, we say the tower representation of $f$ is {\em non-repeating} and furthermore that it has arbitrarily long sequences of consecutive spacers.
\end{enumerate}

Our main result is the following theorem.

\begin{theorem}
\label{thm}
Let $f$ be a rank-1 homeomorphism of a Polish space $X$ with a non-repeating tower representation.  If $g$ is a homeomorphism of $X$ that commutes with $f$, then there is some $i \in \Z$ so that $f^i = g$.

\end{theorem}

Before we proceed with the general analysis that will lead us to the proof of the theorem, we prove a technical proposition about the words $W_n$ that come from a non-repeating tower representation of a rank-1 homeomorphism~$f$.  Let $(\{B_n\}, \{h_n\})$ be a non-repeating tower representation of a rank-1 homeomorphism $f$.  Let $m>n$ and consider $B_m$, $f(B_m)$, \dots, $f^{h_m -1} (B_m)$, the levels of the stage-$m$ tower.  

Condition (3) of the definition of rank-1 implies that each level of the stage-$m$ tower is either a subset of a level of the stage-$n$ tower or is contained in $L_n$.  Since $B_m \subseteq B_n$, level $i$ of the stage-$m$ tower is contained in level $i$ of the stage-$n$ tower, for $0 \leq i < h_n$.  Thus, $W_m$ begins with an occurrence of $W_n$.  

It is easy to see that either $f^{h_n} (B_m) \subseteq B_n$ or $f^{h_n} (B_m) \subseteq L_n$.  Indeed, if $f^{h_n} (B_n) \subseteq f^i (B_n)$ with $0 < i < h_n$, then $f^{h_n - i} (B_m) \subseteq B_n$, which contradicts the fact that $B_n$, $f(B_n)$, \dots, $f^{h_n -1} (B_n)$ are pairwise disjoint.  By similar reasoning, the smallest $j \geq h_n$ for which $f^j (B_m)$ is contained in some level of the stage-$n$ tower is such that $f^j (B_m) \subseteq B_n$.  For this $j$, if $h_n \leq k < j$, then $f^{k} (B_m) \subseteq L_n$.  Thus, the initial $W_n$ in $W_m$ is followed by $(j- h_n)$-many 1s and then another occurrence of $W_n$.  This pattern continues and $W_m$ can be viewed as a disjoint collection of occurrences of $W_n$ interspersed with 1s.  This is described more concretely below.

Let $E_{m,n} = \{i \in [0, h_m) : f^i (B_m) \subseteq B_n \}$.  It is clear that if $i \in E_{m,n}$, then $W_m$ has an occurrence of $W_n$ beginning at position $i$.  Such an occurrence of $W_n$ in $W_m$ is called {\em expected}.  The arguments in the preceding paragraph show that each 0 in $W_m$ is a part of exactly one expected occurrence of $W_n$.  In particular, if $i$ and $j$ are consecutive elements of $E_{m,n}$ and $i + h_n\leq k < j$, then $W_m (k) = 1$.  

It is possible for $W_m$ to have unexpected occurrences of $W_n$.  For example, if $W_{n+1} = W_n W_n 1 W_n W_n$, for all $n>0$, then for $0<n<m$, there is at least one unexpected occurrence of $W_n$ in $W_m$.  However, the following proposition guarantees that knowing a sufficiently large subword of $W_m$ that begins with a specified occurrence of $W_n$ is enough to determine whether that specified occurrence of $W_n$ is expected (and ``large enough" is independent of $m$).  


For the proposition below, and the lemma that follows it, recall that we are working with a rank-1 homeomorphism $f$ of a Polish space $X$ that has a non-repeating tower representation $(\{B_n\}, \{h_n\})$.

\begin{proposition}
\label{prop:expected}

For any $n \in \mathbb{N}$, there is some $l(n) \in \mathbb{N}$ so that for any $m>n$, if $s$ is a subword of length $l(n)$ of $W_m$ that begins with an expected occurrence of $W_n$, then every occurrence of $s$ in $W_m$ begins with an expected occurrence of $W_n$.

\end{proposition}

To prove this proposition, we need the following lemma.

\begin{lemma}
\label{lemma:expected}
Suppose $W_m$ has an expected occurrence of $W_n$ that begins at $i$ and is followed by $r$ 1s and then another expected occurrence of $W_n$.  Suppose further that $i < j < i + h_n$ and that $W_m$ also has a (necessarily unexpected) occurrence of $W_n$ that begins at $j$ that is followed by $s$ 1s and then another occurrence of $W_n$.  Then $r=s$. 
\end{lemma}

\begin{proof}
There are four known occurrences of $W_n$, beginning at $i$, $i + h_n +r$, $j$, and $j + h_n + s$.  Let $\alpha$ be the subword of $W_m$ beginning at $j$ and ending at $i + h_n - 1$; $\alpha$ has length between 1 and $h_n -1$, inclusive.  Let $a$ be the number of 0s in $\alpha$ (thus, $a > 0$).  Let $\beta$ be the subword of $W_m$ beginning at $i + h_n + r$ and ending at $j + h_n -1$; $\beta$ has length between 1 and $h_n -1 - r$, inclusive.  Let $b$ be the number of 0s in $\beta$ (thus, $b > 0 $).  Notice that the occurrence of $W_n$ that begins at $j$ consists exactly of $\alpha 1^r \beta$, so there are exactly $a + b$ 0s in $W_n$.  Notice also that the occurrence of $W_n$ that begins at $i$ ends with $\alpha$ and so the word $W_n$ must end with $\alpha$.  Notice also that the occurrence of $W_n$ that begins at $i + h_n + r$ begins with $\beta$ and so the word $W_n$ must begin with $\beta$.  By counting the number of 0s in $W_n$, we see that $W_n$ can also be expressed as $\beta 1^r \alpha$.  In particular, the $W_n$ that begins at $i + h_n + r$ has the form $\beta 1^r \alpha$ and also the form $\beta 1^s \alpha ^\prime$, where $\alpha^\prime$ is an initial segment of the occurrence of $W_n$ that begins at $j + h_n + s$.  Since $\beta 1^r \alpha = \beta 1^s \alpha ^\prime$ and both $\alpha$ and $\alpha ^\prime$ begin with 0, $r=s$.
\end{proof}

We now give the proof of Proposition \ref{prop:expected}.

\begin{proof}
Suppose $n\in \mathbb{N}$ is such that for each $l \in \mathbb{N}$ there is some $m \in \mathbb{N}$ and a subword $s$ of $W_m$ of length $l$ that begins with an expected occurrence of $W_n$, so that there is also an occurrence of $s$ in $W_m$ that does not begin with an expected occurrence of $W_n$.

Since $W_\infty$ is not periodic, there is some $k > n$ so that the number of 1s that separate expected occurrences of $W_n$ in $W_k$ is not constant.  Let $l = 2h_k + h_n$.  Find $m \in \mathbb{N}$ and $s$ a subword of length $l$ of $W_m$ that begins with an expected occurrence of $W_n$, and so that $W_m$ has an occurrence of $s$ that does not begin with an expected occurrence of $W_n$.  

First consider the occurrence of $s$ in $W_m$ that begins with an expected occurrence of $W_n$.  Because of the regularity with which expected occurrences of $W_n$ appear in $W_m$, $s$ can be written as $W_n 1^{r_1} W_n 1 ^{r_2} \ldots W_n 1^{r_t} A$, where $A$ is a proper initial segment (possibly empty) of $W_n$.

Now consider the occurrence of $s$ in $W_m$ that does not begin with an expected occurrence of $W_n$ (say it begins at position $i$ in $W_m$).  Then there are occurrences of $W_n$ that begin at positions $i$, $i + h_n + r_1$, $i + 2h_n + r_1 + r_2$, \dots, $i + (t-1)h_n + (r_1 + \dots + r_{t-1})$.  It is easy to check that since $W_n$ begins and ends with 0, none of these occurrences of $W_n$ are expected in $W_m$ and, moreover, that each of them intersects exactly two expected occurrences of $W_n$ in $W_m$.  Repeated application of Lemma \ref{lemma:expected} shows that $r_1 = r_2 = \ldots = r_{t-1}$ and, moreover, that if two occurrences of $W_n$ are completely contained in $s$ and separated only by 1s, then they are separated by exactly $r_1$-many 1s.  

Also, since there is an expected occurrence of $W_n$ that contains the 0 at position $i + h_n -1$ and that expected occurrence of $W_n$ ends with 0, we know that $r_1 < h_n$.  This clearly implies that in $s$ there is no consecutive sequence of 1s with length $h_n$.  

Any occurrence of $s$ in $W_m$ begins with a 0, which is a part of some expected occurrence of $W_k$.  The length of $s$ is large enough to guarantee that this occurrence of $s$ will completely contain the next expected occurrence of $W_k$ in $W_m$.  This implies that each pair of consecutive expected occurrences of $W_n$ in $W_k$ are separated by exactly $r_1$-many 1s.  This contradicts the choice of $k$.
\end{proof}

\section{Initial Analysis}
Let $f$ be a rank-1 homeomorphism of a Polish space $X$ with a non-repeating tower representation $(\{B_n\}, \{h_n\})$.

\begin{proposition}
Each level of the stage-$n$ tower is a disjoint union of at least two levels of the stage-$(n+1)$ tower.
\end{proposition}

\begin{proof}
Let $0 \leq i <h_n$ and consider $f^i (B_n)$, level $i$ of the stage-$n$ tower.  If $x\in f^i (B_n)$, then for some $0 \leq j < h_{n+1}$, $x \in f^j (B_{n+1})$, which implies that $f^j (B_{n+1}) \subseteq f^i (B_n)$.  Since the levels of the stage-$(n+1)$ tower are pairwise disjoint, $f^i(B_n)$ is a disjoint union of some levels of the stage-$(n+1)$ tower.  To see that $f^i(B_n)$ contains at least two levels of the stage-$(n+1)$ tower, notice that $B_{n+1} \subsetneq B_n$, which guarantees that $f^i (B_{n+1}) \subsetneq f^i(B_n)$.
\end{proof}

\begin{proposition}
\label{prop:opentower}
Every nonempty open set contains a level of a tower.
\end{proposition}

\begin{proof}
It suffices to show that for each $n\in \N$ and each $i \in \Z$, each of the sets $f^i (B_n)$ and $f^i (L_n)$ contains a level of some tower.  

The image, under $f$, of a non-top level of the stage-$m$ tower is a level of the stage-$m$ tower.  The top of the stage-$m$ tower contains at least two levels of the stage-$(m+1)$ tower, and only one of these can be the top of the stage-$(m+1)$.  Thus, the image, under $f$, of the top of the stage-$m$ tower contains a level of the stage-$(m+1)$ tower.  Similarly, the pre-image under $f$ of any level of the stage-$m$ tower contains a level of the stage-$(m+1)$ tower.  We thus have the following: if $A$ contains a level of a tower, then so does each of $f(A)$ and $f^{-1} (A)$.  It now suffices to show that each $L_n$ contains a level of some tower.

If $m>n$, then each level of the stage-$m$ tower is either contained in $L_n$ or is disjoint from $L_n$.  If for every $m>n$, every level of the stage-$m$ tower is disjoint from $L_n$, then $W_\infty = W_n W_n W_n \dots$, which contradicts the fact that $(\{B_n\}, \{h_n\})$ is a non-repeating tower representation of $f$.  Therefore, $L_n$ contains some level of some tower.
\end{proof}

Let $B_\infty = \bigcap B_n$,  $T_\infty = \bigcap f^{h_n -1}(B_n)$, and $L_\infty = \bigcap L_n$.  We say that a point $x \in X$ is an {\em interior point} (with respect to $f$ and  $(\{B_n, h_n\})$, its tower representation) if no point in the orbit of $x$ is in $B_\infty$, $T_\infty$, or $L_\infty$.  Interior points are relatively simple to deal with and will play a crucial role in the proof of the main theorem.  

\begin{proposition}
\label{prop:interior}
The set of interior points is comeager in $X$.
\end{proposition}

\begin{proof}
It suffices to show that each of the sets $L_\infty$, $B_\infty$, and $T_\infty$ are nowhere dense.  That $L_\infty$ is nowhere dense is an immediate consequence of the previous proposition.  We now prove that $B_\infty$ is nowhere dense.  (The proof that $T_\infty$ is nowhere dense is essentially the same.)

Suppose $U$ is a non-empty open set.  We need to show that there is some non-empty open $V \subseteq U$ that does not intersect $B_\infty$.  By the previous proposition, $U$ contains a level of a tower.  If it is a level that is not the base, then we can set $V$ equal to that level.  If, on the other hand, there is a base of a tower, say $B_n$, contained in $U$, then we know that $B_n$ is the disjoint union of at least two levels of the stage-$(n+1)$ tower.  As only one of those can be the base of the stage-$(n+1)$ tower, another of them can be chosen for $V$.  In any case, we have $V$, a non-base level of a tower.  The clopen set $V$ is disjoint from $B_\infty$.
\end{proof}

\begin{proposition}
\label{prop:interiork}
Let $x$ be an interior point.  Then for each $k \in \mathbb{N}$ there is some $N \in \mathbb{N}$ so that for all $n>N$, there is some $i$ satisfying $k \leq i < h_n - k$ and $x \in f^i (B_n)$.
\end{proposition}

\begin{proof}
Suppose $x$ is an interior point.  Since $x \notin L_\infty$, $x \in f^{j_M} (B_M)$ for some $M \in \mathbb{N}$ and $0 \leq j_M < h_M$.  By condition (2) of the definition of rank-1 homeomorphisms, there is, for each $m \geq M$, a unique $j_m$ satisfying $0 \leq j_m < h_m$ so that $x \in f^{j_m}(B_m)$.  That the sequences $(j_m: m \geq M)$ and $(h_m - j_m: m \geq M)$ are non-decreasing follows from the fact that $B_{m+1} \subseteq B_m$ and $f^{h_{m+1}-1 } (B_{m+1}) \subseteq f^{h_m -1} (B_m)$ for each $m \in \mathbb{N}$.

To prove the proposition, it suffices to show that  neither of these sequences is eventually constant.  But it is easy to see that if $\{j_m\}$ is eventually constantly $c$, then $f^{-c} (x) \in B_\infty$, which contradicts the fact that $x$ is an interior point.  Similarly, if $\{h_m - j_m\}$ is eventually constantly $c$, then $f^{c-1} (x) \in T_\infty$, which contradicts the fact that $x$ is an interior point.
\end{proof}

\begin{corollary}
Let $x$ be an interior point.  Then for each $i$, there is some $n$ so that $f^i (x)$ is in a level of the stage-$n$ tower that is neither the base nor the top.
\end{corollary}

\begin{proposition}
\label{prop:distinctinterior}
If $x$ and $y$ are distinct interior points, then there is some level of some tower that contains exactly one of $x$ and $y$.
\end{proposition}

\begin{proof}
Suppose that $x$ and $y$ are interior points that are contained in the same levels of the same towers.  We must show that $x=y$.  We do this by showing that $x$ and $y$ are in the same basic clopen sets.

Suppose that for some $n \in \mathbb{N}$ and some $k \in \Z$, either $f^k (B_n)$ or $f^k (L_n)$ contains exactly one of $x$ and $y$.  By Proposition \ref{prop:interiork} we can find $m>n$ so that there are $i,j$ satisfying $|k| \leq i,j < h_m - |k|$ so that $x \in f^i (B_m)$ and $y \in f^j (B_m)$.  But, since $x$ and $y$ are in the same levels of the stage-$m$ tower, $i = j$.  Now $f^{-k}(x)$ and $f^{-k} (y)$ are both in the same level of the stage-$m$ tower.  By repeated application of condition (3) of the definition of a rank-1 homeomorphism, either there is some level of the stage-$n$ tower that contains both $f^{-k}(x)$ and $f^{-k}(y)$, or $f^{-k}(x)$ and $f^{-k}(y)$ are both in $L_n$.  Thus, neither $f^k(B_n)$ nor $f^k(L_n)$ contains exactly one of $x$ and $y$.

Therefore, $x$ and $y$ are in the same basic clopen sets and $x=y$.
\end{proof}

\begin{proposition}
\label{prop:open}
If $U$ is a nonempty open set and $x \notin \bigcup _{i \in \Z} f^i(U)$, then $x$ is the unique fixed point of $f$.
\end{proposition}

\begin{proof}
Let $U$ be nonempty and open.  By Proposition \ref{prop:opentower}, $U$ contains some level of some tower.  So, $\bigcup _{i \in \Z} f^i(U)$ contains some $B_n$, which contains $B_m$ for every $m>n$.  It follows that $\bigcup _{i \in \Z} f^i(U)$ contains every level of every tower and thus contains everything that is not in $L_\infty$.  In fact, if $x \notin \bigcup _{i \in \Z} f^i(U)$, then the entire orbit of $x$ is contained in $L_\infty$.  

Suppose that $x \notin \bigcup _{i \in \Z} f^i(U)$.  Thus, for all $k \in \Z$ and all $n \in \N$, both $x$ and $f(x)$ are elements of $f^k (L_n)$ and not elements of $f^k (B_n)$.  So, $x$ and $f(x)$ are in the same basic clopen sets and thus, $x = f(x)$.

Now suppose that $y$ is a fixed point of $f$.  Since no element of any level of any tower is a fixed point, $y \in L_\infty$.  Thus, for all $k \in \Z$ and all $n \in \N$, $y \in f^k (L_n)$ and $y \notin f^k (B_n)$.  Thus, $x$ and $y$ are in the same basic clopen sets, and so, $x=y$.
\end{proof}

\section{Further Analysis}

\subsection{Simplifying Assumptions}

As before, let $f$ be a rank-1 homeomorphism of a Polish space $X$ with a non-repeating tower representation $(\{B_n\}, \{h_n\})$.  Let $g$ be a homeomorphism of $X$ that commutes with $f$.  We will work towards proving that there is some $k \in \Z$ so that $f^k = g$.  For the proof of Lemma \ref{commute} below, we will distinguish between two cases: either the tower representation $(\{B_n\}, \{h_n\})$ has arbitrarily long sequences of consecutive spacers, or it has bounded sequences of consecutive spacers.  If the latter holds, then let $a_{max}$ denote the length of the longest sequence of consecutive spacers in $W_\infty$.  Equivalently, $a_{max}$ is the largest natural number for which there exist $n,m\in \N$ so that $0 \leq m \leq m + a_{max} < h_n$ and so that for each $j$ satisfying $m < j \leq m + a_{max}$, $f^j (B_n) \subseteq L_0$.

Before proceeding further, we make some simplifying assumptions, which we can do without loss of generality.  If the tower representation $(\{B_n\}, \{h_n\})$ has bounded sequences of consecutive spacers, we can assume that $h_0 > a_{max}$.  Indeed, we can modify the witnessing sequences $\{B_n\}$ and $\{h_n\}$ by deleting the initial entry of each sequence and still have a tower representation for $f$ that is non-repeating and has absolutely bounded sequences of consecutive spacers.  Doing this repeatedly will guarantee that the initial element of the height sequence will be larger than $a_{max}$.

For the next two simplifying assumptions, consider $g^{-1} (B_0)$.  Since $g$ is a homeomorphism, this nonempty set is open and thus must contain $f^m (B_n)$, for some $n \in \N$ and $m \in \Z$.  Notice that $gf^m (B_n) \subseteq B_0$ and that $gf^m$ is a homeomorphism of $X$ that commutes with $f$.  If $gf^m$ is an integral power of $f$, then so is $g$.  Thus we may assume that $m = 0$ and thus that $g (B_n) \subseteq B_0$.  Also, we may assume that $n=1$, for otherwise we can delete the entries in the sequences $\{B_n\}$ and $\{h_n\}$ that are indexed by 1 through $(n-1)$, inclusive.  With these two simplifying assumptions we now have the following.  If $x\in B_1$, then $g(x) \in B_0$.

\subsection{The Set Z(x)}

For any $x \in X$, let $$Z(x) = \{i \in \Z : f^i (x) \in B_1\}.$$

The crucial fact is that if $x$ is an interior point, then the set $Z(x)$ contains enough information to recover $x$.

\begin{proposition}
\label{Z}
If $x$ and $y$ are interior points with $Z(x) = Z(y)$, then $x=y$.  
\end{proposition}

\begin{proof}
Suppose that $x$ and $y$ are distinct interior points and that $Z(x) = Z(y)$.  By Proposition \ref{prop:distinctinterior}, some level of some tower contains exactly one of $x$ and $y$.  
It follows from condition (3) of the definition that for sufficiently large $n$ there is a level of the stage-$n$ tower that contains exactly one of $x$ and $y$.  As $x$ and $y$ are both interior points, each of $x$ and $y$ are in some level of the stage-$n$ tower, for sufficiently large $n$.  Choose such an $n$ and let $i$ and $j$ be such that $0 \leq i, j < h_n$ and $x \in f^i (B_n)$ and $y \in f^j (B_n)$.  Without loss of generality, assume $i<j$.

By Proposition \ref{prop:interiork} there is some $N$ so that neither $x$ nor $y$ are in any of the top $l(n)$ levels of the stage-$N$ tower (i.e., for all $0 \leq k <l(n)$, neither $f^k (x)$ nor $f^k(y)$ is in $f^{h_N -1} (B_N)$, the top of the stage-$N$ tower).  Recall that $l(n) \in \mathbb{N}$ is such that if $s$ is subword of $W_N$ of length $l(n)$ that begins with an expected occurrence of $W_n$, then every occurrence of $s$ in $W_N$ begins with an expected occurrence of $W_n$ (see Proposition \ref{prop:expected}).

Since $Z(x) = Z(y)$, we know that for all $k\in \mathbb{Z}$, $f^k (x) \in B_1$ iff $f^k (y) \in B_1$.  It immediately follows from this that if $0 \leq m < h_1$, then for all $k\in \mathbb{Z}$, $f^k (x) \in f^m (B_1)$ iff $f^k (y) \in f^m (B_1)$.  In other words, we know that for all $k$,  $f^k (x)$ is in level $m$ of the stage-$1$ tower iff  $f^k (y)$ is in level $m$ of the stage-$1$ tower.  But each level of the stage-$1$ tower is either completely contained in a level of the stage-0 tower or is disjoint from all levels of the stage-0 tower.  Thus for all $k \in \mathbb{Z}$, 
\begin{equation} 
\label{equation}
f^k (x) \in \bigcup_{0 \leq m < h_0} f^m (B_0) \textnormal{\hspace{.1in} iff \hspace{.1in}} f^k (y) \in \bigcup_{0 \leq m < h_0} f^m (B_0) .
\end{equation}

Now consider the points $f^{-i}(x)$, $f^{-i +1} (x)$, \dots, $f^{-i + l(n)-1} (x)$.  Since $x$ is not in any of the top $l(n)$ levels of the stage-$N$ tower, these points correspond to a subword of length $l(n)$ in $W_N$ that begins with an expected occurrence of $W_n$.  Call this subword $s$.

But now consider the points $f^{-i}(y)$, $f^{-i +1} (y)$, \dots, $f^{-i + l(n)-1} (y)$.  Since $y$ is not in any of the top $l(n)$ levels of the stage-$N$ tower, these points correspond to a subword of length $l(n)$ in $W_N$.  In fact, by equation (1) above, this subword is exactly $s$.  However, since $y$ is on level $j$ of the stage-$n$ tower and $i <j$, the occurrence of $s$ in $W_N$ that corresponds to the points $f^{-i}(y)$, $f^{-i +1} (y)$, \dots, $f^{-i + l(N)-1} (y)$ does not begin with an expected occurrence of $W_n$.  This is a contradiction.
\end{proof}

We will analyze the relationship between $Z(x)$ and $Z(g(x))$, but first we mention two facts and prove a proposition.  It is easy to see that distinct elements of $Z(x)$ cannot be two close to each other.  Indeed, if $f^i (x) \in B_1$, i.e., if $f^i (x)$ is in the base of the stage-1 tower, then $f^{i+1} (x)$ must be in the first level of the stage-1 tower, $f^{i + 2} (x)$ must be in the second level of the stage-1 tower, etc.  If $f^{i + k} (x)$ is again in the base of the stage-1 tower, with $k>0$, then it must be the case that $k \geq h_1$.  We thus have the following fact.

\begin{fact}
For any $x\in X$, if $i\neq i^\prime$ are both in $Z(x)$, then $|i - i^\prime| \geq h_1$.
\end{fact}

More generally, and for similar reasons, we have:

\begin{fact}
For any $x\in X$, if $i \neq i^\prime$ and both $f^i (x)$ and $f^{i^\prime} (x)$ are in the same level of the stage-$n$ tower, then $|i - i^\prime| \geq h_n$.
\end{fact}

\begin{proposition}
\label{prop:unbounded}
If $x$ is an interior point, then $Z(x)$ is neither bounded above nor from below.
\end{proposition}

\begin{proof}
It is clear from the definition of interior point that $x$ is an interior point iff every element of the orbit of $x$ under $f$ is an interior point.  It thus suffices to show that for each interior point $x$, $Z(x)$ contains both a positive and a negative element.

By proposition \ref{prop:interiork} there is some $n$ and some $i$ satisfying $h_1 \leq i < h_n -h_1$ and so that $x \in f^i (B_n)$.  Now $f^{-i} (x) \in B_n \subseteq B_1$, so $-i$ is a negative element of $Z(x)$.  But also, $f^{h_n - i -1} (x)$ is in the top of the stage-$n$ tower and thus is in the top of the stage-1 tower.  So $f^{h_n - i - h_1} (x) \in B_1$ and hence, $h_n - i - h_1$ is a positive element of $Z(x)$.
\end{proof}

\subsection{The Function $\phi _x$}

We now work towards showing a very rigid connection between $Z(x)$ and $Z(g(x))$, as long as $x$ and $g(x)$ are interior points.  

For any $x \in X$, there is a natural way to define a function from $Z(x)$ to $Z(g(x))$.  If $i \in Z(x)$, then $f^i(x) \in B_1$.  This implies that $gf^i (x) \in B_0$.  So there is a unique $m= m(i)$ with $0 \leq m < h_1$ so that $gf^i(x)$ is in level $m$ of the stage-1 tower, i.e., in $f^m (B_1)$.  Now $f^{i-m} g(x) \in B_1$, so $i-m \in Z(g(x))$.  We thus have a function $\phi_x : Z(x) \rightarrow Z(g(x))$, given by $\phi_x(i) = i-m$.  It is clear that for each $i\in Z(x)$, $$i - h_1 < \phi_x(i) \leq i.$$

A priori, it may be the case that $m$ depends on $i$.  The main line of argument in this paper hinges on the fact, shown in Lemma \ref{commute} below, that as long as $x$ and $g(x)$ are both interior points, there is no such dependence. 

 

\begin{lemma}
\label{inj}
For any $x$, the function $\phi_x$ is an order preserving injection.
\end{lemma}

\begin{proof}
Suppose $i$ and $i^\prime$ are distinct elements of $Z(x)$ with $i > i^\prime$.  By Fact 1 above, we have $i - i^\prime \geq h_1$ and so, $i - h_1 \geq i^\prime$.  But we also have $\phi_x (i) > i - h_1$ and $i^\prime \geq \phi_x(i^\prime)$.  Together these give $\phi_x(i) > \phi_x(i^\prime)$.
\end{proof}

In the next lemma, the levels of the stage-$n$ tower that are subsets of $B_1$ will be important.  For $n>0$, let $r_n$ denote the number of such levels.  Since the definition of rank-1 ensures that the top of the stage-$n$ tower is a subset of the top of the stage-1 tower, the highest level of the stage-$n$ tower that is contained in $B_1$ is level $h_n - h_1 $ (for $n>0$).

\begin{lemma}
\label{surj}
If $x$ is an interior point, then the function $\phi_x$ is surjective.
\end{lemma}

\begin{proof}
Let $x$ be an interior point and suppose $j \in Z(g(x)) \setminus \mathrm{rng}  (\phi _x)$.  Consider the point $f^{j + h_1 -1} (x)$.  Since $x$ is interior, we can find some $n >1$ and $0 \leq m < h_n$ so that $f^{j + h_1 -1} (x) \in f^m (B_n)$.  Clearly, $f^{j + h_1 -1 -m} (x)$ is in the base of the stage-$n$ tower.

Now consider the interval $I = [j+ h_1 - 1 -m, j+ h_n -1 -m]$.  The set $\{f^i (x) : i \in I\}$ contains one element from each of the bottom $h_n - h_1 +1$ levels of the stage-$n$ tower.  This includes all of the levels that are subsets of $B_1$.  So $|I \cap Z(x)| = r_n$.

Now consider the interval $J = [j-m, j+h_n -1 -m]$.  If $i \in J \cap Z(g(x))$, then $f^i (g(x)) \in B_1$, and so $f^i (g(x))$ is in some level of the stage-$n$ tower that is contained in $B_1$.  We claim that $|J \cap Z(g(x))|> r_n$.  First, if $i \in I \cap Z(x)$, then, since $i - h_1 < \phi_x(i) \leq i$, $\phi_x (i) \in J \cap Z(g(x))$.  But $j$ is also in $J \cap Z(g(x))$ and, by assumption, $j \notin \mathrm{rng}(\phi_x)$.  Therefore, $|J \cap Z(g(x))| > r_n$.

This implies the existence of distinct $i, i^\prime \in J$ so that both $f^i (g(x))$ and $f^{i^\prime} (g(x))$ are in the same level of the stage-$n$ tower.  By Fact 2 above, $|i - i^\prime| \geq h_n$.  This is impossible, since $J = [j-m, j+h_n -1 -m]$.
\end{proof}

\subsection{The Function $\Psi _x$}

For $x$ an interior point, we define a function $\Psi _x: Z(x) \rightarrow \N$ as follows.  Let $i \in Z(x)$ and find $j>i$ as small as possible so that $j \in Z(x)$.  Let $\Psi_x (i) = j - i $.

\begin{lemma}
\label{commute}
Suppose $x$ and $g(x)$ are interior points.  Then for each $i \in Z(x)$, $\Psi_x (i) = \Psi_{g(x)} (\phi_x (i))$.
\end{lemma}

The proof of Lemma \ref{commute} will be done differently for the two cases.  We first give the proof in the case that the tower representation $(\{B_n\}, \{h_n\})$ has bounded sequences of consecutive spacers.

\begin{proof}
Let $x$ be such that $x$ and $g(x)$ are interior points and let $ i \in Z(x)$.  We want to show that $\Psi_x (i) = \Psi_{g(x)} (\phi_x (i))$.  Let $j$ be the smallest element of $Z(x)$ that is greater than $i$.  We have $\Psi_x (i) = j - i$.  Also, since we are in the case with absolutely bounded sequences of consecutive spacers, we have:
\begin{equation}
0 \leq \Psi_x (i) - h_1 \leq a_{max}
\end{equation}

Since $\phi _x : Z(x) \rightarrow Z(g(x))$ is an order preserving bijection, we have that $\phi_x (j)$ is the largest element of $Z(g(x))$ that is greater than $\phi _x (i)$.  So we have $\Psi_{g(x)} (\phi_x(i)) = \phi_x(j) - \phi_x (i) $.  And, as before, we have:
\begin{equation}
0 \leq \Psi _{g(x)} (\phi_x (i)) - h_1 \leq a_{max}
\end{equation}

Equations (2) and (3) clearly imply: 
\begin{equation}
|\Psi _x (i) - \Psi _{g(x)} (\phi _x (i))| \leq a_{max}
\end{equation}

We will show that in fact, $\Psi_x (i) = \Psi_{g(x)} (\phi_x(i)) $.

First, consider the point $f^i g(x)$.  Since $i \in Z(x)$, $f^i g(x) \in B_0$.  We claim that also $f^{i + \Psi_{g(x)} (\phi_x(i))} g (x) \in B_0$.  Indeed, since $f^i g(x) \in B_0$, the point $f^i g(x)$ must be in some level of the stage-1 tower.  Let $0 \leq m < h_1$ be such that $f^i g(x) \in f^m (B_1)$.  Then $f^{i - m} g(x) \in B_1$ and $\phi _x (i) = i - m$.  Now $f^{i - m + \Psi_{g(x)} (\phi_x(i))} g (x) \in B_1$, and so $f^{i  + \Psi_{g(x)} (\phi_x(i))} g (x) \in f^m (B_1)$.  Since we know that $f^m (B_1)$ intersects $B_0$, it must be contained in $B_0$.  Thus $f^{i + \Psi_{g(x)} (\phi_x(i))} g (x) \in B_0$.

Next, consider the point $f^j g (x)$.  Since $j \in Z(x)$, $f^j g(x) \in B_0$.  But $j =  i + \Psi_x (i)$, so $f^{i + \Psi_x (i)} g(x) \in B_0$.

Suppose that $\Psi_{g(x)} (\phi_x(i)) \neq \Psi_x (i)$.  Then $i +  \Psi_{g(x)} (\phi_x(i)) \neq i  + \Psi_x (i)$.  Fact 2 then implies that $$|(i + \Psi_{g(x)} (\phi_x(i)))-({i + \Psi_x (i)} )| \geq h_0 .$$  Since $h_0 > a_{max}$, we have $$|\Psi_{g(x)} (\phi_x(i))-\Psi_x (i) | > a_{max},$$ which contradicts equation (4) above.
\end{proof}

The proof of Lemma \ref{commute} is more involved in the case that the tower representation $(\{B_n\}, \{h_n\})$ has arbitrarily long sequences of spacers.  In this case, we need to show that $\Psi_x$ and $\Psi_{g(x)}$ exhibit an almost periodic behavior.  Since the elements of $Z(x)$ are not equally spaced, it will be easier to describe this type of periodicity if we first identify $\Z$ with $Z(x)$.  This can be done because, by Proposition \ref{prop:unbounded}, $Z(x)$ is neither bounded from above nor from below.

Fix some $i_0 \in  Z(x)$ and let this correspond to $0 \in \Z$.  This extends uniquely to an order preserving correspondence of $\Z$ with $Z(x)$.  Let $i_k$ denote the element of $Z(x)$ that corresponds to $k \in \Z$.  This correspondence between $\Z$ and $Z(x)$ extends through the order preserving bijection $\phi: Z(x) \rightarrow Z(g(x))$ to a correspondence between $\Z$ and $Z(g(x))$.  For $k \in \Z$, let $j_k$ denote $\phi (i_k) \in Z(g(x))$.  In the ensuing discussion of $\Psi_x$ and $\Psi_{g(x)}$ we will take the domain of each function to be $\Z$; that is, we will write $\Psi_x (k)$ in place of $\Psi_x (i_k)$ and we will write $\Psi_{g(x)} (k)$ in place of $\Psi_{g(x)} (j_k)$.

Recall that for $n>0$, $r_n$ is the number of levels in the stage-$n$ tower that are contained in $B_1$.
\begin{claim}
Let $x$ be an interior point.  
\begin{enumerate}
\item  For each $n>1$, $\Psi_x$ is constant each congruence class mod $r_n$ except one.  On this last congruence class mod $r_n$, $\Psi_x$ is unbounded.
\item  For each $k \in \Z$, there is an $n>1$ so that $\Psi_x$ is constant on the congruence class of $k$ mod $r_n$.
\end{enumerate}
\end{claim}

\begin{proof}
For each $n>0$ we will define $R_n$, a sequence of natural numbers of length $r_n -1$.  Let $z$ be any point in the base of the stage-$n$ tower.   Let $\{j_0, j_1, \ldots, j_{r_n-1}\}$ enumerate, from smallest to largest, the elements of the set $\{j \in [0, h_n -1]: f^j (z) \in B_1\}$.  We now define $R_n$ to be the sequence $(j_1 -j_0 , j_2 - j_1 , \ldots, j_{r_n -1} - j_{r_n -2} )$.  

It is clear that the definition of $R_n$ is independent of which point $z \in B_n$ is chosen.  So whenever $i_k \in Z(x)$ is such that $f^{i_k} (x) $ is in the base of the stage-$n$ tower, there is an occurrence of the word $R_n$ that begins at position $k$ in $\Psi_x$.  But since $x$ is an interior point, if $k$ is such that $f^{i_k} (x) \in B_n$, then both $f^{i_{k + r_n}} (x)$ and $f^{i_{k - r_n}} (x)$ are in the base of the stage-$n$ tower.  So for such a $k$ the function $\Psi_x$ is constant on the congruence class of $k+m$ mod $r_n$, for each $0 \leq m < r_n -1$.  But since we are in the case that the tower representation has unbounded sequences of consecutive spacers, the last congruence class mod $r_n$ must be unbounded.  This proves part 1.

Let $k\in \Z$.  Since $f^{i_k} (x)$ is in the base of the stage-1 tower, $f^{i_k + h_1 -1} (x)$ is in the top of the stage-1 tower.  Since $x$ is an interior point, there is some $n$ so that $f^{i_k + h_1 -1} (x)$ is not in the top of the stage-$n$ tower.  For this $n$, $\Psi_x$ is constant on the congruence class of $k$ mod $r_n$.
\end{proof}

We now give the proof of Lemma \ref{commute} in the case that the tower representation $(\{B_n\}\{h_n\})$ has arbitrarily long sequences of spacers.

\begin{proof}
Let $x$ and $g(x)$ be interior points.  We want to show that $\Psi_x = \Psi_{g(x)}$.  Suppose that $\Psi_x (i) \neq \Psi_{g(x)} (i)$.  Choose $n$ so that $\Psi_{g(x)}$ is constant on the set $\{i + k(r_n): k \in \Z\}$ and let $j$ be such that $\Psi_{g(x))}$ is unbounded on the set $\{j + k(r_n): k \in \Z\}$.  Clearly $i$ and $j$ are in different congruence classes mod $r_n$.  If $\Psi_x$ is unbounded on the set $\{j + k(r_n): k \in \Z\}$, then $\Psi_x$ agrees with $\Psi_{g(x)}$ at each position except perhaps those in $\{j + r(z_n): r \in \Z\}$ (which contradicts the fact that $\Psi_x$ and $\Psi_{g(x)}$ differ at $i$).  Thus $\Psi_x$ is constant on the set $\{j + k(r_n): k \in \Z\}$.

Now find $k \in \Z$ so that $\Psi_{g(x)} (k) - \Psi_x (k) \geq h_1$.  We have that $i_{k+1} - i_k =  \Psi_x (k)$ and also that $j_{k+1}  - j_k =  \Psi_{g(x)} (k)$.  It follows that $$\Psi_{g(x)} (k) - \Psi _x (k) = (j_{k+1} - i_{k+1}) + (i_k - j_k).$$

But we know that $j_{k+1} - i_{k+1} \leq 0$ and that $i_k - j_k < h_1$.  Thus, $\Psi_{g(x)} (k) - \Psi _x (k) < h_1$, a contradiction. 
\end{proof}

\begin{lemma}
\label{power}
If $x$ and $g(x)$ are interior points, then for some $k \in \Z$, $f^k (x) = g(x)$.
\end{lemma}

\begin{proof}
Suppose $x$ and $g(x)$ are interior points.  We know that for each $i \in Z(x)$, $ i - h_1 < \phi_x(i) \leq i$.  

We claim that $i - \phi _x (i)$ is independent of $i$.  Indeed, if $i - \phi_x (i)$ is not independent of $i$, then we can find consecutive $i$ and $j$ in $Z(x)$ so that $i - \phi _x (i) \neq j - \phi_x (j)$.  By saying that $i$ and $j$ are consecutive elements of $Z(x)$, we formally mean that $j \in Z(x)$ is as small as possible satisfying $j >i$.  So we have $$\phi _x (j) - \phi_x (i) \neq j -i.$$  Since $i$ and $j$ are consecutive elements of $Z(x)$, $j-i = \Psi_x (i) $.  Since $\phi_x : Z(x) \rightarrow Z(g(x))$ is an order preserving bijection we also have that $\phi_x (i)$ and $\phi_x (j)$ are consecutive elements of $Z(g(x))$ and so $\phi _x (j) - \phi_x (i) = \Psi _{g(x)} (\phi_x (i))$.  But then $$\Psi _{g(x)} (\phi_x (i)) \neq \Psi_x (i) $$ and this contradicts Lemma \ref{commute}.

Thus, $i - \phi _x (i)$ is independent of $i$.  Let $k$ be such that for all $i \in Z(x)$, $i - \phi_x (i) =k$.  Since $\phi_x : Z(x) \rightarrow Z(g(x))$ is a bijection we have that $i \in Z(x)$ iff $i - k \in Z(g(x))$.  But clearly, $i \in Z(x)$ iff $i - k \in Z(f^k(x))$.  Thus, $Z(f^k(x)) = Z(g(x))$ and, by Proposition \ref{Z}, $f^k(x) = g(x)$.  
\end{proof}

\section{Proof of the Main Theorem}

We now prove Theorem \ref{thm}.

\begin{proof}
Let $f$ be a rank-1 homeomorphism of a Polish space $X$ with a non-repeating tower representation and let $g$ be a homeomorphism of $X$ that commutes with $f$.  By Proposition \ref{prop:interior}, the set of interior points is comeager in $X$.  So the set $\{x \in X : x \textnormal{ and }g(x) \textnormal{ are interior points}\}$ is also comeager in $X$.

Since both $g$ and $f$ are homeomorphisms of $X$, each $A_i = \{x \in X : g(x) = f^i (x)\}$ is closed.  By Lemma \ref{power}, each element of the comeager set $\{x \in X : x \textnormal{ and }g(x) \textnormal{ are interior points}\}$ is in some $A_i$.  Therefore, some $A_i$ is dense in some nonempty open set $U$.  Since $A_i$ is closed, it contains $U$.  Since $g$ and $f$ commute, $A_i$ is invariant under $T$.  By Proposition \ref{prop:open}, $\bigcup _{i \in \Z} f^i(U)$ is either all of $X$ or all of $X$ except the unique fixed point of $f$.  But the unique fixed point of $f$ must be a fixed point of $g$ (since $g$ and $f$ commute) and thus must be in $A_i$.

We now have that $A_i$ is all of $X$.  Thus $g = f^i$.
\end{proof}

\end{document}